\documentclass[a4paper,12pt,twoside]{amsart}
\usepackage[english]{babel}
\usepackage{amssymb, amsmath, eucal, enumerate}
\usepackage[all]{xy}

\usepackage[left=2.5cm,right=2.5cm,top=3cm,bottom=3cm]{geometry}

\usepackage[svgnames,hyperref]{xcolor}
\newcommand{\mycolor}{Navy}
\usepackage[pdfauthor={Son H. Do and  Thanh Cong N. Pham}, colorlinks, linktocpage, citecolor = \mycolor, linkcolor = \mycolor, urlcolor = \mycolor]{hyperref}
\newtheorem{The}{Theorem}[section]
\newtheorem{Lem}[The]{Lemma}
\newtheorem{Prop}[The]{Proposition}
\newtheorem{Cor}[The]{Corollary}
\newtheorem{Rem}[The]{Remark}

\newtheorem{Def}[The]{Definition}

\newcommand{\C}{\mathbb{C}}
\newcommand{\R}{\mathbb{R}}
\newcommand{\N}{\mathbb{N}}

\newcommand{\dt}{\partial_t}
\newcommand{\MAu}{(dd^cu)^n}

\begin{document}
 \title[Parabolic Monge-Amp\`ere equations]{A comparison principle for
 	 parabolic complex Monge-Amp\`ere equations} 
\setcounter{tocdepth}{1}
\author{Hoang-Son Do} 
\address{Institute of Mathematics \\ Vietnam Academy of Science and Technology \\18
Hoang Quoc Viet \\Cau Giay, Hanoi, Vietnam}
\email{hoangson.do.vn@gmail.com, dhson@math.ac.vn}
\author{Thanh Cong Ngoc Pham}
\address{Department of Mathematics\\ VNU University of Science, 334 Nguyen Trai\\ Thanh Xuan, Ha Noi, VietNam}
\email{phamngocthanhcong1997@gmail.com }
\date{\today\\ {\it Keywords:} Viscosity solutions, Parabolic Monge-Amp\`ere equation, pluripotential theory.
\\ The first author was supported by Vietnam Academy of Science and Technology under grant number CT0000.07/21-22.}
\maketitle
\begin{center}
	{\it Dedicated to Professor Ahmed Zeriahi on the occasion of his retirement.}
\end{center}
\begin{abstract}
In this paper, we study the Cauchy-Dirichlet  problem for  Parabolic complex Monge-Amp\`ere equations  on  strongly pseudoconvex domains  using the viscosity method.  We prove a comparison principle for 
Parabolic complex Monge-Amp\`ere equations and use it to study 
the existence and uniqueness of viscosity solution in certain cases where the sets $\{z\in\Omega: f(t, z)=0 \}$
may be pairwise disjoint.
\end{abstract}
\tableofcontents
\section{Introduction}
In Algebraic Geometry, the Minimal Model Program is known as a process of simplifying algebraic varieties through algebraic surgeries in birational geometry. In \cite{ST,ST12},
 Song and Tian gave a conjectural picture to approach the Minimal Model Program via the K\"ahler-Ricci flow.
 This approach requires  a theory of weak solutions for certain  degenerate parabolic complex Monge-Amp\`ere equations.

A  viscosity approach for   parabolic Monge-Amp\`ere (PMA) equations has been developed  by  Eyssidieux-Guedj-Zeriahi both on domains \cite{EGZ15b} and on compact K\"ahler manifolds \cite{EGZ16, EGZ18} (see also \cite{DLT} and
 \cite{To21} for some generalizations). In another direction, a theory of  pluripotential solutions for PMA equations has been
 developed in \cite{GLZ1, GLZ2}. Under suitable conditions, the notions of these weak solutions
 are equivalent \cite{GLZ3}. Besides having applications for the Minimal Model Program, the theories of
 weak solutions for  PMA equations are interesting topics in themself. The aim of this paper is to study the theory of viscosity
 solutions for PMA equations in domains of $\C^n$.

 Let $\Omega\subset\C^n$ be a strongly pseudoconvex domain and let $T\in (0, \infty)$. We consider the following
  Cauchy-Dirichlet problem:
\begin{equation}\label{PMA}
\begin{cases}
\MAu=e^{\dt u+F(t, z, u)}\mu(t, z)\qquad\mbox{ in }\qquad \Omega_T,\\
u=\varphi\qquad\mbox{in}\qquad [0, T)\times\partial\Omega,\\
u(0, z)=u_0(z)\qquad\mbox{in}\quad\bar{\Omega},
\end{cases}
\end{equation}
where
\begin{itemize}
	\item  $\Omega_T=(0, T)\times\Omega$.
	\item $F(t, z, r)$ is continuous in $[0, T]\times\bar{\Omega}\times\R$ and non-decreasing in $r$.
	\item $\mu(t, z)=f(t, z)dV$,
	 where $dV$ is the standard volume form in $\C^n$ and $f\geq 0$ is a bounded  continuous function in $[0, T]\times\Omega$.
	 \item $\varphi(t, z)$ is a continuous function in $[0, T]\times\partial\Omega$.
	 \item $u_0(z)$ is continuous in $\bar{\Omega}$ and plurisubharmonic in $\Omega$ such that $u_0(z)=\varphi(0, z)$ in $\partial\Omega$.
\end{itemize}
In \cite{EGZ15b}, Eyssidieux-Guedj-Zeriahi proved that if $(u_0, \mu(0, z))$ is admissible
(see Definition \ref{def admissible}) and $F, f, \varphi$ do not depend on $t$ then \eqref{PMA} has a unique viscosity solution.
In \cite{DLT}, this result has been extended to the case where $F, f, \varphi$ depend on $t$ and $f$ satisfies
some additional conditions under which  $\{z\in\Omega: f(t, z)=0 \}\subset\{z\in\Omega: f(s, z)=0 \}$ for $0<s<t<T$
 (see \cite[Theorem 4.13]{DLT}). In the general case, with $f$ being merely a non-negative, bounded, continuous function, the
 question about the existence and uniqueness of viscosity solution to \eqref{PMA} is still open.
 
 In this paper, we prove a comparison principle for \eqref{PMA} and use it to study 
  the existence and uniqueness of viscosity solution to \eqref{PMA}  in certain cases where the sets $\{z\in\Omega: f(t, z)=0 \}$
 may be pairwise disjoint.
  Specifically, we assume that $\Phi: (-1, 1)\times\Omega\rightarrow\C^n$ is a continuous mapping satisfying the following conditions:
 \begin{itemize}
 	\item the mapping $z\mapsto\Phi(s, z)$ is holomorphic in $\Omega$ for every $s\in (-1, 1)$; 
 	\item $\Phi (0, z)=z$ for every $z\in\Omega$;
 	\item $\partial_t\Phi$ is well defined and continuous on $(-1, 1)\times\Omega$.
 	In particular, $\partial_t\Phi$ is holomorphic on $\{0\}\times\Omega$ and,
 	for each $U\Subset\Omega$, there exist $C_U>0$ and $\delta_U>0$ such that
 	$\Phi (s, z)\in\Omega$ and
 	\begin{equation}\label{eq Phi}
 	|\Phi (s, z)-z|\leq C_U|s|,
 	\end{equation}
 	 for all $(s, z)\in (-\delta_U, \delta_U)\times U$.
 \end{itemize}
 Our main result is as follows:
\begin{The}\label{Main Theorem} 
Suppose that the following conditions are satisfied
\begin{itemize}
	\item for every $0<R<S<T$ and $K\Subset\Omega$, there exist $a, b>0$ such that if $(t_0, z_0)\in (R, S)\times K$ and $f(t_0, z_0)=0$ then
	\begin{equation}\label{eq0.0 Main Theorem}
	f(t_0, z)\leq\exp\Big(b-a\dfrac{|\langle z-z_0, \partial_t\Phi (0, z_0)\rangle|}{|z-z_0|^2} \Big),
	\end{equation}
	for every $z\in\Omega\setminus\{z_0\}$;
	\item for every $0<R<S<T$, $K\Subset\Omega$ and $\epsilon>0$, there exists $0<\delta<\delta_K$ such that
	\begin{equation}\label{eq0 Main Theorem}
	(1+\epsilon)f(t, z)\geq f(t+s, \Phi(s, z)),
	\end{equation}
	for every $z\in K$, $R<t<S$ and $|s|<\delta$.
\end{itemize}
  Assume that $u$ and $v$, respectively, is a bounded viscosity subsolution and
a bounded viscosity supersolution to \eqref{PMA}. Then, for every $0<R<S<T$, $K\Subset\Omega$ and $\epsilon>0$,
 there exists
 $0<\delta<\delta_K$ such that 
\begin{center}
	$u(t+s_1, \Phi(s_1, z))<v(t+s_2, \Phi(s_2, z))+\epsilon,$
\end{center}
for all $z\in K$, $R<t<S$ and $\max\{|s_1|, |s_2|\}<\delta$.
\end{The}
It is easy to see that if $f$ does not depend on $t$ then $f$ satisfies
\eqref{eq0.0 Main Theorem} and \eqref{eq0 Main Theorem} with $\Phi(s, z)=z$.
 Some other simple examples
are $(f, \Phi)=(g(tz_0+z), -sz_0+z)$ and $(f, \Phi)=(g(e^{it}z), e^{-is}z)$, where 
$z_0\in \C^n$ and $g(z)=e^{-1/(h(z))^2}$ for some Lipschitz function $h$.
 If $f_1, f_2$ satisfy \eqref{eq0.0 Main Theorem} and \eqref{eq0 Main Theorem} for
the same $\Phi$ then $tf_1+(T-t)f_2$ satisfies \eqref{eq0.0 Main Theorem} and \eqref{eq0 Main Theorem}.

We expect that Theorem \ref{Main Theorem} still holds without the condition \eqref{eq0.0 Main Theorem}. We need this condition for some estimates in the proof of Lemma \ref{lem compa}.

 By using Theorem \ref{Main Theorem}, we obtain the
following result:
\begin{Cor}\label{Main Cor}
	Assume that $(u_0, \mu(0, z))$ is admissible
	(see Definition \ref{def admissible}). Suppose that $\Phi$ and $f$
	satisfy the conditions in Theorem \ref{Main Theorem}. Then \eqref{PMA} has a unique viscosity solution.
\end{Cor}
\section{Preliminaries}
In this section, we recall some basic concepts and well-known results about viscosity sub/super-solutions. The reader can find more
details in \cite{Jen88}, \cite{Ish89}, \cite{IL90}, \cite{CIL92}, \cite{EGZ} and \cite{DLT}.
\begin{Def} (Test functions)
	Let  $ w : \Omega_T \longrightarrow \R$ be any function defined in $\Omega_T$ and $(t_0,z_0) \in \Omega_T$ a given point. 
	An upper test function (resp. a lower test function) for $w$ at the point $(t_0,z_0)$ 
	is  a $C^{(1,2)}$-smooth function $q$ (i.e., $\partial_tq$, $D_zq$ and $D^2_zq$ are continuous on
	on the domain of $q$)
	in a neighbourhood of the point $(t_0,z_0)$ such that $  w (t_0,z_0) = q (t_0,z_0)$ and $w \leq q$ (resp. $w \geq q$) in a neighbourhood of $(t_0,z_0)$.
\end{Def}
\begin{Def}
	1. A  function  $u\in USC(\Omega_T)$ is said to be a (viscosity) subsolution to the parabolic complex Monge-Amp\`ere equation 
	\begin{equation}\label{PMAfree}
	\MAu=e^{\dt u+F(t, z, u)}\mu(t, z),
	\end{equation}
	 in $\Omega_T$ if for any point $ (t_0,z_0) \in \Omega_T$ and any upper test function $q$ for $u$ at $(t_0,z_0)$, we have
	$$
	(dd^c q_{t_0} (z_0))^n  \geq e^{\partial_t q (t_0,z_0) + F (t_0,z_0,q (t_0,z_0))} \mu (t_0, z_0).
	$$
	In this case, we also say that $u$ satisfies the differential inequality 
	$$(dd^c u)^n \geq e^{\partial_t u (t,z) + F (t,z,u (t,z))} \mu(t, z),$$ in the viscosity sense in $\Omega_T.$
	
	A function $u\in USC([0, T)\times\overline{\Omega})$ is called  a subsolution to  the Cauchy-Dirichlet problem \eqref{PMA} if $u$ is a subsolution to \eqref{PMAfree} 
	satisfying $u\leq\varphi$ in $[0, T)\times\partial\Omega$ and
	$u(0, z)\leq u_0(z)$ for all $z\in\Omega$.
	
	\medskip
	2. A  function  $ v\in\Omega_T$ is said to be a (viscosity) 
	supersolution to the parabolic complex Monge-Amp\`ere equation \eqref{PMAfree}
	in $\Omega_T$ if for any point 
	$ (t_0,z_0) \in \Omega_T$ and any lower test  function $q$ for $v$ at $(t_0,z_0)$ such that $dd^c q_{t_0}(z_0) \geq 0$, we have 
	$$
	(dd^c q_{t_0})^n (z_0) \leq e^{\partial_{t} q (t_0,z_0) + F (t_0,z_0,q (t_0,z_0))} \mu (t_0, z_0).
	$$
	In this case we also say that $v$ satisfies the differential inequality 
	$$(dd^c v)^n \leq e^{\partial_t v (t,z) + F (t,z,v(t,z))} \mu(t, z),$$
	 in the viscosity sense in $\Omega_T$.
	
	A function $v\in LSC([0, T)\times\overline{\Omega})$ is called  a supersolution to \eqref{PMA} if $v$ is a supersolution to \eqref{PMAfree}
	satisfying $v\geq\varphi$ in $[0, T)\times\partial\Omega$ and
	$v(0, z)\geq u_0(z)$ for all $z\in\Omega$.
	
	\medskip
	3. A function $u$ is said to be a (viscosity) solution to \eqref{PMAfree}
	(respectively, \eqref{PMA}) if it is a subsolution and a supersolution to
 \eqref{PMAfree} (respectively, \eqref{PMA}). 
\end{Def}
\begin{Rem}\label{rem subsolution}
	a) By the same argument as in the proof of \cite[Proposition 1.3]{EGZ}, if $u$ is a subsolution 
	to \eqref{PMAfree}
	and  $q$ is an upper test function for $u$ at $(t_0,z_0)\in\Omega_T$ then $dd^c q_{t_0} (z_0)\geq 0$;\\
	b) If  $u$ is a subsolution 
	to \eqref{PMAfree} then $u(t, z)$ is plurisubharmonic in $z$ for every $t\in (0, T)$ (see \cite[Corollary 3.7]{EGZ15b}).
\end{Rem}
 Denote by $\mathcal{S}_{2n}$ the space of all $2n\times 2n$ symmetric matrices.
 For each function  $ u : \Omega_T \longrightarrow \R$ and for every $(t_0,z_0) \in \Omega_T$, 
we define by 	$\mathcal P^{2,+} u (t_0,z_0)$ the set of $(\tau, p, Q)\in\R\times \R^{2n}\times\mathcal{S}_{2n}$ satisfying
 \begin{equation} 
 	u (t,z)  \leq u (t_0,z_0) +  \tau (t-t_0) + o (\vert t-t_0\vert) 
  + \langle p, z - z_0\rangle + \frac{1}{2} \langle Q (z - z_0), z-z_0\rangle   
 	+ o (\vert z - z_0\vert^2),
 \end{equation}
and denote by $\bar{\mathcal P}^{2,+} u (t_0,z_0)$ the set of 
$(\tau, p, Q)\in\R\times \R^{2n}\times\mathcal{S}_{2n}$ satisfying: $\exists (t_m, z_m)\rightarrow (t_0, z_0)$
	and $(\tau_m, p_m, Q_m)\in \mathcal P^{2,+} u (t_0,z_0)$ such that $(\tau_m, p_m, Q_m)\rightarrow (\tau, p, Q)$ and
	$u(t_m, z_m)\rightarrow u(t_0, z_0)  \}.$\\

	We define in the same way  the  sets  $\mathcal P^{2,-} u (t_0,z_0)$ and $\bar{\mathcal P}^{2,-} u (t_0,z_0)$
 by 
	$$
	\mathcal P^{2,-} u (t_0,z_0) = - \mathcal P^{2,+} (-u) (t_0,z_0),
	$$  
	and
	$$
	\bar{\mathcal P}^{2,-} u (t_0,z_0)=-\bar{\mathcal P}^{2,+}(-u)(t_0,z_0).
	$$ 
	Since $F$ and $f$ are continuous, by \cite[Proposition 2.6]{EGZ}, we have: 
\begin{Prop} \label{prop def vis}
	\text{ }
	
	1. An upper semi-continuous function $u : \Omega_T \longrightarrow \R$ is a subsolution to the parabolic equation 
		\begin{equation}\label{PMA prop P}
	\MAu=e^{\dt u+F(t, z, u)}\mu(t, z),
	\end{equation}
	 if and only if for all
	$ (t_0,z_0) \in \Omega_T$ and  $(\tau,p,Q) \in 	\bar{\mathcal P}^{2,+} u (t_0,z_0),$  we have
	$dd^cQ\geq 0$ and
	\begin{equation} \label{eq:sub}
(dd^c Q)^n\geq	e^ {\tau +  F (t_0,z_0, u (t_0,z_0))} \mu (t_0,z_0).
	\end{equation}
	Here $dd^cQ:=(dd^c \langle Q z, z\rangle)$, $z\in\C^n=\R^{2n}$.\\
	2. A lower semi-continuous function $v : \Omega_T \longrightarrow \R$  is 
	a supersolution to the parabolic equation \eqref{PMA prop P}  if and only if  for all
	$ (t_0,z_0) \in \Omega_T$ and  $(\tau,p,Q) \in 	\bar{\mathcal P}^{2,-} u (t_0,z_0)$ such that $dd^c Q \geq 0,$ we have
	\begin{equation} \label{eq:super}
(dd^c Q)^n\leq	e^ {\tau +  F (t_0, z_0,v (t_0,z_0))} \mu (t_0,z_0).
	\end{equation}
\end{Prop}
The following theorem is the parabolic Jensen-Ishii’s maximum principle which plays an important role in the theory of viscosity solution:
\begin{The}\cite[Theorem 8.3]{CIL92}\label{the maximal}
	Let  $u\in USC(\Omega_T)$ and $v\in LSC(\Omega_T)$.
	Let $\phi$ be a function defined in $(0, T) \times \Omega^2$ such that $(t, \xi ,\eta) \longmapsto \phi (t,\xi, \eta)$ is continuously differentiable in $t$ and twice continuously differentiable in $(\xi ,\eta)$.  
	
	Assume that the function $(t,\xi, \eta) \longmapsto u (t,\xi) - v(t,\eta) - \phi (t,\xi, \eta)$ has a local maximum at some point $(\hat t, \hat\xi, \hat\eta) \in (0, T) \times \Omega^2$. 
	
	Assume furthermore that both $w=u$ and $w=-v$ satisfy:
	\begin{displaymath}
	(\label{Cond}\ref{Cond})\left\{
	\begin{array}{ll}
	\forall (s,z) \in \Omega & \exists r>0 \ \text{such that} \ \forall M >0 \  \exists C \ \text{satisfying} \\
	&\left.
	\begin{array}{l}
	|(t, \xi)-(s,z)| \le r,\\
	(\tau,p,Q)\in \mathcal{P}^{2,+}w(t, \xi) \\
	|w(t, \xi)|+|p| + |Q| \le M
	\end{array} \right\} \Longrightarrow \tau\le C.
	\end{array}\right.
	\end{displaymath}
	
	Then for any $\kappa > 0$, there exists 
	$(\tau_1,p_1,Q^+) \in \bar{\mathcal P}^{2,+} u (\hat t, \hat \xi)$, 
	$(\tau_2,p_2,Q^-) \in \bar{\mathcal P}^{2,-} v (\hat t, \hat \eta)$ such that
	$$\tau_1 = \tau_2 + D_t \phi (\hat t, \hat \xi,\hat \eta), \ p_1 = D_{\xi} \phi (\hat t, \hat\xi,\hat\eta),
	 \ p_2 = - D_{\eta} \phi (\hat t, \hat\xi,\hat\eta)$$ and
	$$
	-\left(\frac{1}{\kappa} + \| A \| \right) I \leq 
	\left(
	\begin{array}{cc}
	Q^+ &0 \\
	0 & - Q^-
	\end{array}
	\right) \leq  A  + \kappa A^2,
	$$
 where $A := D_{\xi, \eta}^2 \phi (\hat t, \hat\xi,\hat\eta)\in\mathcal{S}_{4n}$. 
\end{The}
The following lemma is deduced by combining Proposition \ref{prop def vis} and Theorem \ref{the maximal}:
\begin{Lem}\label{lem inf sup}
Let $(u_{\tau})$ be a   locally uniformly bounded family of real valued functions defined in $\Omega_T$. 

1. Assume that for every $\tau$, $u_{\tau}$ is a viscosity subsolution to the equation
\begin{equation}\label{eq Lem liminf limsup}
(dd^c w)^n=e^{ \partial_t w + F(t,z,w)} \mu(t, z) ,
\end{equation}
 in $\Omega_T$.
Then  
$\overline u= (\sup_{\tau} u_{\tau})^*$
 is a subsolution to \eqref{eq Lem liminf limsup}
in $\Omega_T$. Here $(\sup_{\tau} u_{\tau})^*$ is the upper semicontinuous regularization
of $\sup_{\tau} u_{\tau}$.

2. Assume that for every $\tau$, $u_{\tau}$ is a viscosity supersolution to \eqref{eq Lem liminf limsup}. Then 
$\underline u = (\inf_{\tau} u_{\tau})_*$
 is a supersolution to \eqref{eq Lem liminf limsup} in $\Omega_T$.
 Here $(\inf_{\tau} u_{\tau})_*$ is the lower semicontinuous regularization
 of $\inf_{\tau} u_{\tau}$.

3. If $\tau\in\N$ then 1. and 2. hold for 	$\overline u= (\limsup\limits_{\tau\to\infty} u_{\tau})^*$
and $\underline u = (\liminf\limits_{\tau\to\infty} u_{\tau})_*$.
\end{Lem}
In the theory of viscosity solution, the comparison principle and Perron method are two key tools for studying
the existence and uniqueness of solution. The following comparison principle has been established in \cite{EGZ15b}:
\begin{The}\label{compa.the}
	\cite[pages 949-953]{EGZ15b}
	Let $u$  (resp. $v$) be a bounded subsolution (resp. supersolution) to the parabolic complex Monge–Amp\`ere equation
	\eqref{PMAfree} in $\Omega_T$.  Assume that one of the following conditions
	is satisfied
	\begin{itemize}
		\item [a)] $\mu (t, z)>0$ for every $(t, z)\in (0, T)\times\Omega$.
		\item[b)] $\mu$ is independent of $t$.
		\item[c)] Either $u$ or $v$ is locally Lipschitz in $t$ uniformly in $z$.
	\end{itemize}
	Then
	\begin{center}
		$\sup\limits_{\Omega_T}(u-v)\leq\sup\limits_{\partial_P(\Omega_T)}(u-v)_+,$
	\end{center}
where $\partial_P(\Omega_T)=(\{0\}\times\overline{\Omega})\cup ((0, T)\times\partial\Omega)$ is
	the parabolic boundary of $\Omega_T$ and
$u$ (resp. $v$) has been extended as an upper (resp. a lower) semicontinuous
function to $\overline{\Omega_T}$.
\end{The}
In order to recall the Perron method, we need the concepts of $\epsilon$-sub/super-barrier.
\begin{Def}
	a) A function $u\in USC([0, T)\times\bar{\Omega})$ is called $\epsilon$-subbarrier for
	\eqref{PMA} if $u$ is subsolution to \eqref{PMAfree}
	in the viscosity sense such that $u_0-\epsilon\leq u_*\leq u\leq u_0$ in 
	$\{0\}\times\bar{\Omega}$ and $\varphi-\epsilon\leq u_*\leq u\leq \varphi$ in
	$[0, T)\times\partial\Omega.$\\
	b) A function $u\in LSC([0, T)\times\bar{\Omega})$ is called $\epsilon$-superbarrier for
	\eqref{PMA} if $u$ is supersolution to \eqref{PMAfree}
	in the viscosity sense such that $u_0+\epsilon\geq u^*\geq u\geq u_0$ in 
	$\{0\}\times\bar{\Omega}$ and $\varphi+\epsilon\geq u^*\geq u\geq \varphi$ in
	$[0, T)\times\partial\Omega.$
\end{Def}
\begin{Prop}\label{prop.subbarrier}\cite[Proposition 4.1]{DLT}
	For all $\epsilon>0$,
	there exists a continuous $\epsilon$-subbarrier for \eqref{PMA}
	which is Lipschitz in $t$. 
\end{Prop}
\begin{Def}\label{def admissible}
	We say that $(u_0, \mu (0, .))$ is
	{\sl  admissible} if for all $\epsilon>0$, there
	exist $u_{\epsilon}\in C(\bar{\Omega})$ and $C_{\epsilon}>0$ such that $u_0\leq u_{\epsilon}\leq u_0+\epsilon$ and $(dd^c u_{\epsilon})^n\leq 
	e^{C_{\epsilon}}\mu(0, z)$ in the viscosity sense.  
\end{Def}
\begin{Prop}\cite[Theorem 1.3]{DLT}
	If $(u_0, \mu (0, .))$ is  admissible then the function
	 $u_\epsilon$ in the definition \ref{def admissible} can be taken to be psh in $\Omega$.
\end{Prop}
\begin{Prop}\label{prop.superbarrier}\cite[Proposition 4.3]{DLT}
	If $(u_0 (z), \mu(0, z))$ is  
	admissible then for all $\epsilon>0$,
	there exists a continuous $\epsilon$-superbarrier for \eqref{PMA}
	which is Lipschitz in $t$. 
\end{Prop}
\begin{Lem}\label{lem perron}(Perron method)\cite[Lemma 2.12]{DLT}
	Assume that for every $\epsilon>0$, the problem \eqref{PMA} admits 
	a continuous $\epsilon$-superbarrier which is Lipschitz in $t$
	and a continuous $\epsilon$-subbarrier. 
	Denote by $S$ the family of all continuous subsolutions  to	\eqref{PMA}.
	Then $\phi_S=\sup\{v: v\in S\}$ is a
	discontinuous viscosity solution to \eqref{PMA}, i.e., $(\phi_S)^*$ is a subsolution and $(\phi_S)_*$
	is a supersolution.
\end{Lem}
\section{Some useful lemmas}
Throughout this section, we always suppose that $\Phi$ and $f$ satisfy the conditions as in Theorem \ref{Main Theorem}. 
Given a bounded function $u:\Omega_T\rightarrow \R$ and a constant $A>2 osc_{\Omega_T}(u)$. For every relatively compact open subet
$U$ of $\Omega$ and for each constant $0<\delta\ll 1$ satisfying $\Phi ([-\delta, \delta]\times U)\subset\Omega)$, we 
define
\begin{center}
	$u^k(t, z)=\sup\{u(t+s, \Phi(s, z))-k|s|: |s|\leq\dfrac{A}{k} \}$,
\end{center}
and
\begin{center}
		$u_k(t, z)=\inf\{u(t+s, \Phi(s, z))+k|s|: |s|\leq\dfrac{A}{k} \}$,
\end{center}
for every $k>\max\{\dfrac{A}{\delta}, \dfrac{2A}{T}\}$ and $(t, z)\in (A/k, T-A/k)\times U$.

We have the following modified version of \cite[Lemma 3.5]{EGZ15b}:
\begin{Lem}\label{lem regularization sub}
	Assume that $u$ is a bounded upper semicontinuous function in $\Omega_T$. Then 
	\begin{itemize}
		\item[(i)] $u^k$ is upper semicontinuous in
		$(A/k, T-A/k)\times U$;
		\item[(ii)]  for all $(t, z)\in (A/k, T-A/k)\times U$,
		\begin{center}
				$u(t, z)\leq u^k(t, z)\leq \sup_{|s|\leq A/k}u(t+s, \Phi(s, z));$
		\end{center}
		\item[(iii)] if $(t, z)$ and $(t+s, \Phi (s, z))$ belong in $(A/k, T-A/k)\times U$ then
		\begin{center}
			$|u^k(t, z)-u^k(t+s, \Phi (s, z))|\leq k|s|;$
		\end{center}
	\item[(iv)] if $\MAu\geq e^{\dt u+F(t, z, u)}\mu(t, z)$ in the viscosity sense in $\Omega_T$ then,
	 for every $0<\epsilon<1$, there exists $k_{\epsilon}>0$ such that, for every $k>k_{\epsilon}$, 
	 \begin{equation}\label{eq0 lem regu}
	 (dd^c u^k)^n\geq (1-\epsilon)e^{\dt u^k+F_k(t, z, u^k)}f_k(t, z)dV,
	 \end{equation}
	in the viscosity sense in $(\delta, T-\delta)\times U$,
	where $F_k(t, z, r)=\inf_{|s|\leq A/k}F(t+s, \Phi(s, z), r)$ and $f_k(t, z)=\inf_{|s|\leq A/k}f(t+s, \Phi(s, z))$.
	\end{itemize}
\end{Lem}
\begin{proof}
	(i) Let $(t_0, z_0)\in  (A/k, T-A/k)\times U$. We will show that
	\begin{center}
		$u^k(t_0, z_0)\geq \limsup\limits_{(t, z)\to (t_0, z_0)}u^k(t, z).$
	\end{center}
Assume that $(t_m, z_m)\in  (A/k, T-A/k)\times U$ satisfies $(t_m, z_m)\rightarrow (t_0, z_0)$
as $m\to\infty$ and
\begin{center}
	$\limsup\limits_{(t, z)\to (t_0, z_0)}u^k(t, z)=\lim\limits_{m\to\infty}u^k(t_m, z_m).$
\end{center}
Since $u$ is usc and $\Phi$ is continuous, by the definition of $u^k$, we have
\begin{center}
	$u^k(t_m, z_m)=u(s_m+t_m, \Phi(s_m, z_m))-k|s_m|,$
\end{center}
for some $|s_m|\leq A/k$. Let $\{s_{m_l}\}$ be a subsequence of $\{s_m\}$ 
which converges to a point $s_0\in [-A/k, A/k]$. Then
\begin{flushleft}
		$\begin{array}{ll}
\limsup\limits_{(t, z)\to (t_0, z_0)}u^k(t, z)&=\lim\limits_{m_l\to\infty}u^k(t_{m_l}, z_{m_l})\\
&=\lim\limits_{m_l\to\infty}(u(s_{m_l}+t_{m_l}, \Phi(s_{m_l}, z_{m_l}))-k|s_{m_l}|)\\
&=\lim\limits_{m_l\to\infty}u(s_{m_l}+t_{m_l}, \Phi(s_{m_l}, z_{m_l}))-k|s_0|\\
&\leq u(s_0+t_0, \Phi(s_0, z_0))-k|s_0|\\
&\leq u^k(t_0, z_0).
	\end{array}$
\end{flushleft}
Hence, $u^k$ is usc in $(A/k, T-A/k)\times U$.\\
(ii) Obvious.\\
(iii) Let $|s_0|\leq A/k$ such that $u^k(t, z)=u(t+s_0, \Phi(s_0, z))-k|s_0|$. If $|s-s_0|>A/k$ then
\begin{flushleft}
	$\begin{array}{ll}
	u^k(t, z)=u(t+s_0, \Phi(s_0, z))-k|s_0| &\leq u(t+s, \Phi(s, z))+2 osc_{\Omega_T}u-k|s_0|\\
	&\leq u^k(t+s, \Phi(s, z))+A-k|s_0|\\
	&\leq  u^k(t+s, \Phi(s, z))+k|s-s_0|-k|s_0|\\
	&\leq  u^k(t+s, \Phi(s, z))+k|s|.
	\end{array}$
\end{flushleft}
If $|s-s_0|\leq A/k$ then 
\begin{flushleft}
	$\begin{array}{ll}
u^k(t, z)=u(t+s_0, \Phi(s_0, z))-k|s_0|&\leq u^k(t+s, \Phi(s, z))+k|s-s_0|-k|s_0|\\
&\leq  u^k(t+s, \Phi(s, z))+k|s|.
	\end{array}$
\end{flushleft}
Hence
\begin{center}
	$u^k(t, z)-u^k(t+s, \Phi(s, z))\leq k|s|.$
\end{center}
By the same argument, we also have
\begin{center}
	$u^k(t+s, \Phi(s, z))-u^k(t, z)\leq k|s|.$
\end{center}
Therefore
\begin{center}
		$|u^k(t, z)-u^k(t+s, \Phi(s, z))|\leq k|s|.$
\end{center}
(iv) Let $r_0>0$ such that $V:=U+r_0\mathbb{B}^{2n}\Subset\Omega$. 
Since $\Phi(t, z)$ is holomorphic in $z$ and converges locally uniformly to $Id$ as $t\rightarrow 0$, we have
$\frac{\partial\Phi_{\alpha}}{\partial z_{\beta}}(t, z)$ converges uniformly in $V$ to $\delta_{\alpha\beta}$ as  $t\rightarrow 0$ for every
$1\leq\alpha, \beta\leq n$. 
Hence, for every $0<\epsilon<1$, there exists $0<r_1<\delta$ such that $\Phi([-r_1, r_1]\times V)\Subset\Omega$ and
\begin{equation}\label{eq1 lem regu}
\Big|\det \left(\dfrac{\partial\Phi_j}{\partial z_k}(t, z)\right)\Big|^2>1-\epsilon,
\end{equation}
for every $(t, z)\in [-r_1, r_1]\times V$. Denote $k_{\epsilon}=\max\{A/r_1, 2A/T\}$. We will show that \eqref{eq0 lem regu}
holds in the viscosity sense in $(\delta, T-\delta)\times U$ for every $k>k_{\epsilon}$.

Let $(t_0, z_0)\in (\delta, T-\delta)\times U$, $s_0\in (-A/k, A/k)$ and let $q$ be an upper test function of
 $u_{s_0}(t, z):=u(t+s_0, \Phi (s_0, z))-k|s_0|$ at $(t_0, z_0)$. Then $\hat{q}(t, z):=q(t-s_0, \Phi^{-1}(s_0, z))+k|s_0|$
 is an  upper test function of $u$ at $(\hat{t}, \hat{z})=(t_0+s_0, \Phi(s_0, z_0))$. Since $\MAu\geq e^{\dt u+F(t, z, u)}\mu(t, z)$ 
 in the viscosity sense, we have
 \begin{equation}\label{eq2 lem regu}
 	(dd^c\hat{q}(\hat{t}, \xi))^n|_{\xi=\hat{z}}\geq 
 	e^{\dt\hat{q}(\hat{t}, \hat{z})+F(\hat{t}, \hat{z}, \hat{q}(\hat{t}, \hat{z}))}\mu(\hat{t}, \hat{z}).
 \end{equation}
 Note that $\dt\hat{q}(\hat{t}, \hat{z})=\dt q(t_0, z_0)$ and
 \begin{center}
 	$(dd^c q(t_0, \xi))^n|_{\xi=z_0}=\left|\det \left(\dfrac{\partial\Phi_j}{\partial z_k}(t, z)\right)\right|^2(dd^c\hat{q}(\hat{t}, \xi))^n|_{\xi=\hat{z}}.$
 \end{center}
Therefore, by \eqref{eq1 lem regu} and \eqref{eq2 lem regu}, we have
\begin{flushleft}
	$\begin{array}{ll}
(dd^c q(t_0, \xi))^n|_{\xi=z_0}&\geq (1-\epsilon)e^{\dt q(t_0, z_0)+F(\hat{t}, \hat{z}, \hat{q}(\hat{t}, \hat{z}))}\mu(\hat{t}, \hat{z})\\
&\geq (1-\epsilon)e^{\dt q(t_0, z_0)+F(\hat{t}, \hat{z}, q(t_0, z_0))}\mu(\hat{t}, \hat{z})\\
&\geq (1-\epsilon)e^{\dt q(t_0, z_0)+F_k(t_0, z_0, q(t_0, z_0))}f_k(t_0, z_0)dV.
	\end{array}$
\end{flushleft}
Since $(t_0, z_0)$ and $q$ are arbitrary, we get $u_{s_0}$ is a subsolution to the equation
\begin{equation}\label{eq3 lem regu}
(dd^c w)^n= (1-\epsilon)e^{\dt w+F_k(t, z, w)}f_k(t, z)dV,
\end{equation}
 in
 $(\delta, T-\delta)\times U$. Then, it follows from Lemma \ref{lem inf sup} that the function
 $$u^k=\sup_{|s_0|\leq A/k}u_{s_0}=(\sup_{|s_0|\leq A/k}u_{s_0})^*$$ is a
  subsolution to \eqref{eq3 lem regu} in
 $(\delta, T-\delta)\times U$.
 
 The proof is completed.
\end{proof}
By the same argument, we have
\begin{Lem}\label{lem regularization super}
	Assume that $u$ is a bounded lower semicontinuous function in $\Omega_T$. Then 
	\begin{itemize}
		\item[(i)] $u_k$ is lower semicontinuous in	$(A/k, T-A/k)\times U$;
		\item[(ii)] for all $(t, z)\in (A/k, T-A/k)\times U$,
		\begin{center}
			$u(t, z)\geq u_k(t, z)\geq \inf_{|s|\leq A/k}u(t+s, \Phi(s, z));$
		\end{center}
		\item[(iii)] if $(t, z)$ and $(t+s, \Phi (s, z))$ belong in $(A/k, T-A/k)\times U$ then
		\begin{center}
			$|u_k(t, z)-u_k(t+s, \Phi (s, z))|\leq k|s|;$
		\end{center}
	\item[(iv)] if $\MAu\leq e^{\dt u+F(t, z, u)}\mu(t, z)$ in the viscosity sense in $\Omega_T$ then,
	 for every $0<\epsilon<1$, there exists $k_{\epsilon}>0$ such that, for every $k>k_{\epsilon}$, 
	$$(dd^c u_k)^n\leq (1+\epsilon)e^{\dt u_k+F^k(t, z, u_k)}f^k(t, z)dV,$$ in the viscosity sense in $(\delta, T-\delta)\times U$,
	where $F^k(t, z, r)=\sup_{|s|\leq A/k}F(t+s, \Phi(s, z), r)$ and $f^k(t, z)=\sup_{|s|\leq A/k}f(t+s, \Phi(s, z))$.
	\end{itemize}
\end{Lem}
In \cite{EGZ15b}, by applying the maximum principle, Eyssidieux-Guedj-Zeriahi have proved the comparison principle for the case where
either the given subsolution or the given supersolution is Lipschitz in $t$. Using the same method as in \cite{EGZ15b}, we obtain
the following lemma:
\begin{Lem}\label{lem compa}
Suppose that $f$ satisfies the conditions in Theorem \ref{Main Theorem}.
	Let $u\in USC\cap L^{\infty} ([0, T)\times\overline{\Omega})$ 
	and  $v\in LSC\cap L^{\infty} ([0, T)\times\overline{\Omega})$ be, respectively, a subsolution and a supersolution to the equation
	 \begin{equation}\label{PMA lem compa}
	(dd^cw)^n=e^{\dt w+F(t, z, w)}\mu(t, z),
	 \end{equation}
	 in $\Omega_T$. Assume that, for $w=u, v$, the following condition holds: for every $U\Subset\Omega$ 
	 and $0<\delta<T/2$, there exists $k(U, \delta)>0$
such that if  $(t, z) \in (\delta, T-\delta)\times U$ then
\begin{equation}\label{eq0 lem compa}
|w(t, z)-w(t+s, \Phi (s, z))|\leq k(U, \delta)|s|,
\end{equation}
	for $0<|s|\ll 1$.
	Then
	\begin{center}
		$\sup\limits_{\Omega_T}(u-v)\leq \sup\limits_{\partial_P\Omega_T}(u-v)_+.$
	\end{center}
\end{Lem}
\begin{proof}
Let $\delta>0$ be an arbitrary positive constant and denote 
\begin{center}
	$h(t, z)=u(t, z)-v(t, z)-\dfrac{\delta}{T-t}+\delta(|z|^2-C),$
\end{center}
where $C=\sup\limits_{z\in\Omega}|z|^2$.
We will show that
\begin{equation}\label{eq1 lem compa}
\max\limits_{\overline{\Omega_T}}h\leq 
\max\limits_{\partial_P\Omega_T}h_+.
\end{equation}
Assume that \eqref{eq1 lem compa} is false. Then, there exists $(t_0, z_0)\in\Omega_T$ such that
\begin{center}
	$M:=h(t_0, z_0)=\max\limits_{[0, T)\times\overline{\Omega}}h>
	\max\limits_{\partial_P\Omega_T}h_+.$
\end{center}
It follows from \cite[Lemma 4.1]{EGZ15b} that
\begin{equation}\label{eq1.1 lem compa}
f(t_0, z_0)=0.
\end{equation}
For every $N>0$, we denote
\begin{center}
	$h_N(t, \xi, \eta)=u(t, \xi)+\delta(|\xi|^2-C)-v(t, \eta)-\dfrac{\delta}{T-t}
	-\dfrac{N(|\xi-z_0|^2+|\eta-z_0|^2)}{2}-\dfrac{N^2|t-t_0|^2}{2},$
\end{center}
and let $(t_N, \xi_N, \eta_N)\in [0, T)\times\overline{\Omega}^2$ such that
\begin{center}
	$h_N(t_N, \xi_N, \eta_N)=\max\limits_{[0, T)\times\overline{\Omega}^2}h_N=:M_N.$ 
\end{center}
By \cite[Proposition 3.7]{CIL92}, we have 
\begin{equation}\label{eq1.2 lem compa}
	\lim_{N\to\infty}(N^2|t_N-t_0|^2+N|\xi_N-z_0|^2+N|\eta_N-z_0|^2)=0.
\end{equation}
 In particular, there exists $N_0>0$
 such that $(t_N, \xi_N, \eta_N)\in (0, T)\times\Omega^2$ for all $N\geq N_0$.
 
 By Lemma \ref{lem 4 weak compa} below, the functions $w_1=u+\delta(|z|^2-C)$ and $w_2=-v$ satisfy the condition \eqref{Cond} 
 in Theorem \ref{the maximal}. Then, it follows from Theorem \ref{the maximal} that, for every $N>N_0$, 
 there exist $(\tau_{N1}, p_{N1}, Q_N^+)\in \overline{\mathcal{P}}^{2,+}w_1(t_N, \xi_N)$ and 
 $(\tau_{N2}, p_{N2}, Q_N^-)\in \overline{\mathcal{P}}^{2,-}v(t_N, \eta_N)$ such that
 \begin{equation}\label{eq2 lem compa}
 	p_{N1}=N(\xi_N-z_0),\qquad p_{N2}=-N(\eta_N-z_0),
 \end{equation}
 and $Q_N^-\geq Q_N^+$ (i.e., $\langle Q_N^-\zeta, \zeta\rangle\geq \langle Q_N^+\zeta, \zeta\rangle$
  for every $\zeta\in\R^{2n}$). In particular,
 we have
 \begin{equation}\label{eq3 lem compa}
 dd^cQ_N^-\geq dd^cQ_N^+\geq \delta\omega>0,
 \end{equation}
 where $\omega=dd^c|z|^2$.
 The  second  inequality holds due to Proposition \ref{prop def vis}.
Moreover, 
it follows from Proposition \ref{prop def vis} that
\begin{equation}\label{eq5 lem compa}
e^{\tau_{N2} +  F (t_N, \eta_N, v (t_N ,\eta_N))} \mu (t_N, \eta_N)  \geq (dd^c Q_N^-)^n.
\end{equation}
Combining  \eqref{eq3 lem compa} and \eqref{eq5 lem compa},
we get
\begin{center}
$e^ {\tau_{N2} +  F (t_N, \eta_N, v (t_N ,\eta_N))} \mu (t_N, \eta_N)
\geq \delta^n\omega^n.$	
\end{center}
Since $F(t, z, v(t, z))$ is bounded, it follows that there exists $m>0$ such that
\begin{equation}\label{eq6 lem compa}
e^ {\tau_{N2}}f(t_N, \eta_N)
\geq m,
\end{equation}
for every $N\gg 1$.

Since $f(t_0, z_0)=0$, the condition \eqref{eq0 Main Theorem} follows that $f(t_N, z_N)=0$, 
where $z_N=\Phi(t_N-t_0, z_0)$. Then, by the condition \eqref{eq0.0 Main Theorem}, there exist $a, b>0$ such that
\begin{equation}\label{eq7 lem compa}
f(t_N, \eta_N)\leq\exp\Big(b-a\dfrac{|\langle z_N-\eta_N, \partial_t\Phi(0, \eta_N)\rangle|}{|z_N-\eta_N|^2}\Big),
\end{equation}
for every $N\gg 1$.

By the assumption on $\Phi$, there exists $C_1>0$ such that
\begin{equation}\label{eq8 lem compa}
|z_N-z_0|\leq C_1|t_N-t_0|=o\Big(\dfrac{1}{N}\Big),
\end{equation}
where the last estimate holds due to \eqref{eq1.2 lem compa}. Combining \eqref{eq8 lem compa} and
\eqref{eq1.2 lem compa}, we get that
\begin{equation}\label{eq8.1 lem compa}
\lim\limits_{N\to\infty}N|z_N-\eta_N|^2=0.
\end{equation}
By \eqref{eq2 lem compa} and by Lemma \ref{lem 4 weak compa} below, we have
\begin{equation}\label{eq9 lem compa}
\tau_{N2}=O(1)-N\langle\eta_N-z_0, \partial_t\Phi (0, \eta_N) \rangle.
\end{equation}
Combining \eqref{eq9 lem compa} and \eqref{eq8 lem compa}, we get
\begin{equation}\label{eq9.1 lem compa}
\tau_{N2}=O(1)-N\langle\eta_N-z_N, \partial_t\Phi (0, \eta_N) \rangle.
\end{equation}
Combining \eqref{eq7 lem compa}, \eqref{eq9.1 lem compa} and \eqref{eq8.1 lem compa}, we get
\begin{center}
	$e^ {2\tau_{N2}}f(t_N, \eta_N)=O(1).$
\end{center}
Since $\lim_{N\to\infty}f(t_N, \eta_N)=f(t_0, z_0)=0$, it follows that 
\begin{center}
	$\lim\limits_{N\to\infty}e^ {2\tau_{N2}}f^2(t_N, \eta_N)=0,$
\end{center}
and it contradicts with \eqref{eq6 lem compa}. Then
\begin{center}
	$\max\limits_{\overline{\Omega_T}}h\leq 
	\max\limits_{\partial_P\Omega_T}h_+.$
\end{center}
Letting $\delta\searrow 0$, we obtain
\begin{center}
	$\sup\limits_{\Omega_T}(u-v)\leq \sup\limits_{\partial_P\Omega_T}(u-v)_+.$
\end{center}
The proof is completed.
\end{proof}
\begin{Lem}\label{lem 4 weak compa}
Let $w$ be a bounded usc function in $\Omega_T$ satisfying the following condition:
for every $U\Subset\Omega$ 
and $0<\delta<T/2$, there exists $k(U, \delta)>0$
such that if  $(t, z) \in (\delta, T-\delta)\times U$ then
\begin{equation}\label{eq0 lem 4 weak compa}
|w(t, z)-w(t+s, \Phi (s, z))|\leq k(U, \delta)|s|,
\end{equation}
for $0<|s|\ll 1$. Then $w$ satisfies the condition \eqref{Cond} in Theorem \ref{the maximal}.
Moreover, 
\begin{equation}\label{eq0.1 lem4 weak compa}
|\tau-\langle p, \partial_t\Phi (0, z_0)\rangle|\leq k(U, \delta),
\end{equation}
for every  $(t_0, z_0) \in (\delta, T-\delta)\times U$ and $(\tau, p, Q)\in \mathcal{P}^{2, +}w(t_0, z_0).$
\end{Lem}
\begin{proof}
	Assume that $q$ is an upper test function for $w$ at $(t_0, z_0)\in (\delta, T-\delta)\times U$.
	By \eqref{eq0 lem 4 weak compa}, for $0<|s|\ll 1$, we have
	\begin{center}
		$q(s+t_0, \Phi(s, z_0))-q(t_0, z_0)\geq -k(U, \delta)|s|.$
	\end{center}
Then, for $0<|s|\ll 1$,
\begin{center}
	$\Big|\dfrac{q(s+t_0, \Phi(s, z_0))-q(t_0, z_0)}{s} \Big|\leq k(U, \delta).$
\end{center}
Letting $s\rightarrow 0$, we get
\begin{center}
	$\Big|\dfrac{\partial q}{\partial t}(t_0, z_0)-\langle \nabla q (t_0, z_0), \partial_t\Phi (0, z_0)\rangle\Big|
	\leq k(U, \delta).$
\end{center}
Note that $(\tau, p, Q)\in \mathcal{P}^{2, +}w(t_0, z_0)$ iff there exists  an upper test function $q$ for $w$ at $(t_0, z_0)$
such that  $(\tau, p, Q)=(\partial_t q(t_0, z_0), Dq(t_0, z_0), D^2q(t_0, z_0))$.
Then, \eqref{eq0.1 lem4 weak compa} is satisfied, and it is easy to see that $w$ satisfies the condition \eqref{Cond} in Theorem \ref{the maximal}.
\end{proof}
\section{Proof of Theorem \ref{Main Theorem} and Corollary \ref{Main Cor}}
For the reader's convenience, we recall Theorem \ref{Main Theorem}:
\begin{The}
	Let $\Phi: (-1, 1)\times\Omega\rightarrow\C^n$ be a continuous mapping satisfying the following conditions:
	\begin{itemize}
		\item the mapping $z\mapsto\Phi(s, z)$ is holomorphic in $\Omega$ for every $s\in (-1, 1)$; 
		\item $\Phi (0, z)=z$ for every $z\in\Omega$;
		\item $\partial_t\Phi$ is well defined and continuous on $(-1, 1)\times\Omega$.
	\end{itemize}
Suppose that $f$ satisfies the following conditions:
\begin{itemize}
	\item for every $0<R<S<T$ and $K\Subset\Omega$, there exist $a, b>0$ such that if $(t_0, z_0)\in (R, S)\times K$ and $f(t_0, z_0)=0$ then
	\begin{equation}
	f(t_0, z)\leq\exp\Big(b-a\dfrac{|\langle z-z_0, \partial_t\Phi (0, z_0)\rangle|}{|z-z_0|^2} \Big),
	\end{equation}
	for every $z\in\Omega\setminus\{z_0\}$;
	\item for every $0<R<S<T$, $K\Subset\Omega$ and $\epsilon>0$, there exists $0<\delta<\delta_K$ such that
	\begin{equation}\label{eq1 main}
	(1+\epsilon)f(t, z)\geq f(t+s, \Phi(s, z)),
	\end{equation}
	for every $z\in K$, $R<t<S$ and $|s|<\delta$.
\end{itemize}
 Assume that $u$ and $v$, respectively, is a bounded viscosity subsolution and
	a bounded viscosity supersolution to the Cauchy-Dirichlet problem \eqref{PMA}. Then, for every $0<R<S<T$, $K\Subset\Omega$ and
	 $\epsilon>0$, there exists $0<\delta\ll 1$ such that
\begin{equation}\label{eq2 main}
	u(t+s_1, \Phi(s_1, z))<v(t+s_2, \Phi(s_2, z))+\epsilon,
\end{equation}
for all $z\in K$, $R<t<S$ and $\max\{|s_1|, |s_2|\}<\delta$.
\end{The}
\begin{proof}
	Let  $0<R<S<T$ and denote $S_1=(S+T)/2$.
	First, we show that for every $\epsilon>0$ there exists
	$\delta_1=\delta_1(\epsilon)\in (0, \min\{ 1, R\})$ such that $K\Subset \Omega_{\delta_1}$ and
	\begin{equation}\label{eq1 proofmain}
	u(t, \eta)<v(s, \xi)+\epsilon,
	\end{equation}
	for every $(t, \eta), (s, \xi)\in \Omega_{S_1}\setminus ([\delta_1, S_1]\times\Omega_{\delta_1})$ with $|t-s|+|\eta-\xi|<\delta_1$.
	Here $\Omega_{\delta_1}=\{z\in\Omega: d(z, \partial\Omega)>\delta_1 \}.$
	
	We consider the mapping
	\begin{center}
		$G: \overline{\Omega_{S_1}}\times\overline{\Omega_{S_1}}\longrightarrow \R$\\
		$(t, \eta, s, \xi)\mapsto u(t, \eta)-v(s, \xi).$
	\end{center}
Since $G$ is upper semicontinuous, the set
\begin{center}
	$U=\{(t, \eta, s, \xi)\in \overline{\Omega_{S_1}}^2: G(t, \eta, s, \xi)<\epsilon \},$
\end{center}
is relatively open in $\overline{\Omega_{S_1}}^2$. Denote
\begin{center}
	$A=\{(t, \eta, t, \eta): (t, \eta)\in ([0, S_1]\times\partial\Omega)\cup(\{0\}\times\overline{\Omega}) \}$.
\end{center}
We have $A$ is compact and $G\leq 0$ on $A$. Hence there exists $0<\delta_1<\min\{ 1, R, dist(K, \partial\Omega)\}$ 
such that
\begin{equation}\label{eq2' proofmain}
	\overline{\Omega_{S_1}}^2\cap (3\delta_1\mathbb{B}_{4n+2}+A)\subset U,
\end{equation}
where $\mathbb{B}_{4n+2}$ is the unit ball in the Euclidean space $\R^{4n+2}$. Then, we have
\begin{equation}\label{eq2 proofmain}
\{(t, \eta, t, \eta): (t, \eta)\in \Omega_{S_1}\setminus ([\delta_1, S_1]\times\Omega_{\delta_1})\}\subset 2\delta_1\mathbb{B}_{4n+2}+A.
\end{equation}
Note that if $|t-s|+|\eta-\xi|<\delta_1$ then $(t, \eta, s, \xi)-(t, \eta, t, \eta)\in \delta_1\mathbb{B}_{4n+2}$.
Therefore, it follows from \eqref{eq2 proofmain} that
\begin{equation}\label{eq2'' proofmain}
\{(t, \eta, s, \xi)\in (\Omega_{S_1}\setminus ([\delta_1, S_1]\times \Omega_{\delta_1}))^2: |t-s|+|\eta-\xi|<\delta_1\}
 \subset 3\delta_1\mathbb{B}_{4n+2}+A.
\end{equation}
Combining \eqref{eq2' proofmain} and \eqref{eq2'' proofmain}, we obtain \eqref{eq1 proofmain}. By the condition
$0<\delta_1<\min\{ 1, R, dist(K, \partial\Omega)\}$, we also have $K\Subset\Omega_{\delta_1}$.

By the assumption \eqref{eq1 main}, there exists 
$\delta_2=\delta_2(\epsilon)\in (0, \frac{\min\{\delta_1, T-S\}}{4})$ such that 
\begin{equation}\label{eq3' proofmain}
8(C_{\Omega_{\delta_1/2}}+1)\delta_2<\delta_1,
\end{equation}
and
\begin{equation}\label{eq3 proofmain}
(1+\epsilon)f(t, z)\geq f(t+s, \Phi(s, z)),
\end{equation}
for every $z\in\Omega_{\delta_1/2}$, $\delta_1/2<t<S_1$ and $|s|<2\delta_2$.
Here $C_{\Omega_{\delta_1/2}}>0$ is defined by \eqref{eq Phi2}. 
 Since $F$ is continuous, we can choose $\delta_2$ small enough such that
\begin{equation}\label{eq4 proofmain}
|F(t, z, r)-F(t+s, \Phi(s, z), r)|<\epsilon,
\end{equation}
for every $z\in\Omega_{\delta_1/2}$, 
$\delta_1/2<t<S_1$, $|s|<2\delta_2$ and $|r|<M$, where $M=\max\{\sup_{\Omega_T}|u|, \sup_{\Omega_T}|v| \}.$

For every $k>\dfrac{2M}{\delta_2}$ and $(t, z)\in [\delta_1/2, S_1]\times\Omega_{\delta_1/2}$, we consider
\begin{center}
	$u^k(t, z)=\sup\{u(t+s, \Phi(s, z))-k|s|: |s|\leq\dfrac{2M}{k} \}$,
\end{center}
and
\begin{center}
	$v_k(t, z)=\inf\{v(t+s, \Phi(s, z))-k|s|: |s|\leq\dfrac{2M}{k} \}$.
\end{center}
It follows from \eqref{eq1 proofmain} and \eqref{eq3' proofmain} that
\begin{equation}\label{eq5' proofmain}
u^k(t+\tau_1, \Phi(\tau_1, z))\leq v_k(t+\tau_2, \Phi(\tau_2, z))+\epsilon,
\end{equation}
for every $(t, z)\in\partial_P((3\delta_1/4, S)\times\Omega_{3\delta_1/4})$ and $\max\{|\tau_1|, |\tau_2|\}<\delta_2$.
Denote $u^k_{\tau_1}(t, z)=u^k(t+\tau_1, \Phi(\tau_1, z))$ and $v_{k, \tau_2}(t, z)=v_k(t+\tau_2, \Phi(\tau_2, z))$.
By Lemma \ref{lem regularization sub} and Lemma \ref{lem regularization super}, there exists
$k_{\epsilon}>\dfrac{2M}{\delta_2}$ such that
 \begin{equation}\label{eq5 proofmain}
(dd^c u^k_{\tau_1})^n\geq (1-\epsilon)e^{\dt u^k_{\tau_1}+F_k(t+\tau_1, \Phi(\tau_1, z), u^k_{\tau_1})}f_k(t+\tau_1, \Phi(\tau_1, z))dV,
\end{equation}
and 
\begin{equation}\label{eq6 proofmain}
(dd^c v_{k, \tau_2})^n\leq (1+\epsilon)e^{\dt v_{k,\tau_2}+F^k(t+\tau_2, \Phi(\tau_2, z), v_{k,\tau_2})}
f^k(t+\tau_2, \Phi(\tau_2, z))dV,
\end{equation}
in the viscosity sense in $(3\delta_1/4, S)\times\Omega_{3\delta_1/4}$ for every $k>k_{\epsilon}$
and $\max\{|\tau_1|, |\tau_2|\}<\delta_2$. Here,
$F_k(t, z, r)=\inf_{|s|\leq 2M/k}F(t+s, \Phi(s, z), r)$, $f_k(t, z)=\inf_{|s|\leq 2M/k}f(t+s, \Phi(s, z))$,
$F^k(t, z, r)=\sup_{|s|\leq 2M/k}F(t+s, \Phi(s, z), r)$ and $f^k(t, z)=\sup_{|s|\leq 2M/k}f(t+s, \Phi(s, z))$.

Moreover, it follows from \eqref{eq3 proofmain} and \eqref{eq4 proofmain} that
\begin{equation}\label{eq7 proofmain}
(1+\epsilon)f_k(t+\tau_1, \Phi(\tau_1, z))\geq f(t, z)\geq \dfrac{f^k(t+\tau_2, \Phi(\tau_2, z))}{1+\epsilon},
\end{equation}
and
\begin{equation}\label{eq8 proofmain}
F_k(t+\tau_1, \Phi(\tau_1, z), r)+\epsilon\geq F(t, z, r)\geq F^k(t+\tau_2, \Phi(\tau_2, z), r)-\epsilon,
\end{equation}
for every $(t, z)\in (3\delta_1/4, S)\times\Omega_{3\delta_1/4}$, $\max\{|\tau_1|, |\tau_2|\}<\delta_2$,
 $|r|\leq M$ and $k>k_{\epsilon}$.

Combining \eqref{eq5 proofmain}, \eqref{eq6 proofmain}, \eqref{eq7 proofmain} and \eqref{eq8 proofmain}, 
we get
\begin{equation}
(dd^c u^k_{\tau_1})^n\geq (1-\epsilon)^2e^{\dt u^k_{\tau_1}+F(t, z, u^k_{\tau_1})-\epsilon}f(t, z)dV,
\end{equation}
and
\begin{equation}
(dd^c v_{k, \tau_2})^n\leq (1+\epsilon)^2e^{\dt v_{k, \tau_2}+F(t, z, v_{k, \tau_2})+\epsilon}f(t, z)dV,
\end{equation}
in the viscosity sense in $(3\delta_1/4, S)\times\Omega_{3\delta_1/4}$ for every $k>k_{\epsilon}$ and $\max\{|\tau_1|, |\tau_2|\}<\delta_2$.
 Therefore, $w_1:=u^k_{\tau_1}+3\epsilon t$
and $w_2:=v_{k, \tau_2}-3\epsilon t$ is respectively a subsolution and a supersolution to the equation
\begin{center}
	$e^{\dt w+F(t, z, w)}\mu(t, z)=(dd^cw)^n,$
\end{center}
in $(3\delta_1/4, S)\times\Omega_{3\delta_1/4}$. Note that, by Lemma \ref{lem regularization sub} and Lemma
\ref{lem regularization super}, the functions $w_1$ and $w_2$ satisfy the condition \eqref{eq0 lem compa} in Lemma \ref{lem compa}.
Then, by using Lemma \ref{lem compa}, we have
\begin{center}
	$\sup\limits_{(3\delta_1/4, S)\times\Omega_{3\delta_1/4}}(w_1-w_2)\leq
	\sup\limits_{\partial_P((3\delta_1/4, S)\times\Omega_{3\delta_1/4})}(w_1-w_2)_+
	\leq \epsilon+6\epsilon S,$
\end{center}
where the last inequality holds due to \eqref{eq5' proofmain}. Then
\begin{center}
	$u^k(t+\tau_1, \Phi(\tau_1, z))-v_k(t+\tau_2, \Phi(\tau_2, z))\leq -6\epsilon t+\epsilon+6\epsilon S\leq (6S+1)\epsilon$,
\end{center}
in $(3\delta_1/4, S)\times\Omega_{3\delta_1/4}$ for every $k>k_{\epsilon}$ and $|\tau|<\delta_2$. Letting $k\rightarrow\infty$, we get
\begin{center}
	$u(t+\tau_1, \Phi(\tau_1, z))-v(t+\tau_2, \Phi(\tau_2, z))\leq (6S+1)\epsilon$
\end{center}
in $(R, S)\times K\subset (3\delta_1/4, S)\times\Omega_{3\delta_1/4}$ 
for every $\max\{|\tau_1|, |\tau_2|\}<\delta_2$. Choosing $\delta=\delta(\epsilon)=\delta_2(\epsilon/(6S+1))$,
we obtain \eqref{eq2 main}.

The proof is completed.
\end{proof}
\begin{Cor}
	Assume that $\Omega$ is a smooth strictly pseudoconvex domain and $(u_0, \mu(0, z))$ is admissible. Suppose that $\Phi$ and $f$
	satisfy the conditions in Theorem \ref{Main Theorem}. Then the Cauchy-Dirichlet problem \eqref{PMA} has a unique viscosity solution.
\end{Cor}
\begin{proof}
	By Propositions \ref{prop.subbarrier} and \ref{prop.superbarrier}, for every $\epsilon>0$, the problem \eqref{PMA} admits 
	a continuous $\epsilon$-superbarrier which is Lipschitz in $t$
	and a continuous $\epsilon$-subbarrier. Then, it follows from Lemma \ref{lem perron} that 
	\begin{center}
		$u:=\sup\{v: v$ is a continuous subsolution to \eqref{PMA}$\}$,
	\end{center}
	is a discontinuous solution to \eqref{PMA}, i.e., $u^*$ is a subsolution and $u=u_*$ is a supersolution.
	
	Let $(t, z)\in\Omega_T$ be an arbitrary point.
	 By Theorem \ref{Main Theorem},
	for every  $\epsilon>0$, there exists $0<\delta\ll 1$ such that $(t+s, \Phi(s, z))\in\Omega_T$ and
	\begin{center}
		$u(t, z)+\epsilon>u^*(t+s, \Phi(s, z))\geq u(t+s, \Phi(s, z)),$
	\end{center}
	for every $|s|<\delta$. Then $\limsup_{s\to 0}u(t+s, \Phi(s, z)\leq u(t, z)+\epsilon$. Letting $\epsilon\searrow 0$, we get
	$\limsup_{s\to 0}u(t+s, \Phi(s, z)\leq u(t, z)$. Therefore, since $u$ is lower semicontinuous, we have
	\begin{equation}\label{eq1 proof main cor}
		\lim_{s\to 0}u(t+s, \Phi(s, z))= u(t, z).
	\end{equation}
	Moreover, it follows from Theorem \ref{Main Theorem} that for every $\epsilon>0$
	 there exists $0<\delta\ll 1$ such that $(t+s, \Phi(s, z))\in\Omega_T$ and
	\begin{center}
		$u(t+s, \Phi(s, z))+\epsilon>u^*(t, z),$
	\end{center}
	for every $|s|<\delta$. Then
	\begin{equation}\label{eq2 proof main cor}
		\liminf_{s\to 0}u(t+s, \Phi(s, z))+\epsilon\geq u^*(t, z).
	\end{equation}
	Combining \eqref{eq1 proof main cor} and \eqref{eq2 proof main cor}, we have
	\begin{center}
		$u(t, z)+\epsilon\geq u^*(t, z).$
	\end{center}
	Letting $\epsilon\searrow 0$, we obtain $u(t, z)\geq u^*(t, z)$, and then $u=u_*=u^*$. Hence, $u$ is a viscosity solution
	to \eqref{PMA}. 
	
	Now, we assume that $u_1$ and $u_2$ are two viscosity solutions to \eqref{PMA}. By 
	Theorem \ref{Main Theorem},
	for every $(t, z)\in\Omega_T$ and  $\epsilon>0$, there exists $0<\delta\ll 1$ such that $(t+s, \Phi(s, z))\in\Omega_T$ and
	\begin{center}
		$u_1(t, z)+\epsilon> u_2(t+s, \Phi(s, z)),$
	\end{center}
	for every $|s|<\delta$. Since $u_2$ is continuous, it implies that
	\begin{center}
			$u_1(t, z)+\epsilon\geq \lim\limits_{s\to 0}u_2(t+s, \Phi(s, z))=u_2(t, z).$
	\end{center}
	Letting $\epsilon\searrow 0$, we get $u_1(t, z)\geq u_2(t, z)$. By the same argument, we also have $u_1(t, z)\leq u_2(t, z)$.
	Then $u_1\equiv u_2$.
	
	Thus \eqref{PMA} admits a unique viscosity solution.
\end{proof}
\section*{Acknowledgements}
	The second-named author would like to thank Vingroup Innovation Foundation (VINIF) for supporting his Master studies at 
	VNU University of Science, Hanoi.

\end{document}